\documentclass[oneside,english,reqno,11pt]{amsart}
\PassOptionsToPackage{backgroundcolor=green!40}{todonotes}
\usepackage{amstext}
\usepackage{enumitem}
\usepackage{amsthm}
\usepackage{libertineRoman}
\usepackage{helvet}
\usepackage{libertineMono}
\usepackage[libertine]{newtxmath}
\usepackage[T1]{fontenc}
\usepackage{geometry}
\usepackage{color}
\usepackage{babel}
\usepackage{mathrsfs}
\usepackage{mathtools}
\usepackage{mleftright}
\mleftright
\usepackage[unicode=true,pdfusetitle,
 bookmarks=true,bookmarksnumbered=false,bookmarksopen=false,
 breaklinks=true,pdfborder={0 0 0},pdfborderstyle={},backref=false,colorlinks=true]
 {hyperref}
\hypersetup{
 allcolors=blue}

\DeclareFontFamily{U}{matha}{\hyphenchar\font45}
\DeclareFontShape{U}{matha}{m}{n}{
      <5> <6> <7> <8> <9> <10> gen * matha
      <10.95> matha10 <12> <14.4> <17.28> <20.74> <24.88> matha12
      }{}
\DeclareSymbolFont{matha}{U}{matha}{m}{n}
\DeclareFontSubstitution{U}{matha}{m}{n}

\DeclareFontFamily{U}{mathx}{\hyphenchar\font45}
\DeclareFontShape{U}{mathx}{m}{n}{
      <5> <6> <7> <8> <9> <10>
      <10.95> <12> <14.4> <17.28> <20.74> <24.88>
      mathx10
      }{}
\DeclareSymbolFont{mathx}{U}{mathx}{m}{n}
\DeclareFontSubstitution{U}{mathx}{m}{n}

\DeclareMathDelimiter{\vvvert}{0}{matha}{"7E}{mathx}{"17}

\hypersetup{final}
\usepackage{aligned-overset}
\usepackage[capitalize,nosort,noabbrev]{cleveref}
\Crefname{assumption}{Assumption}{Assumptions}
\Crefname{prop}{Proposition}{Propositions}
\Crefname{thmstepnvi}{Step}{Steps}
\Crefname{thmstepnvii}{Step}{Steps}
\Crefname{lem}{Lemma}{Lemmas}
\Crefname{thm}{Theorem}{Theorems}
\Crefname{prop}{Proposition}{Propositions}
\Crefname{fig}{Figure}{Figures}

\crefdefaultlabelformat{#2\textup{#1}#3}
\crefformat{equation}{\textup{(#2#1#3)}}
\crefmultiformat{equation}{(#2#1#3)}%
{ and~(#2#1#3)}{, (#2#1#3)}{ and~(#2#1#3)}
\crefrangeformat{equation}{(#3#1#4--#5#2#6)}

\numberwithin{equation}{section}
\numberwithin{figure}{section}
\numberwithin{table}{section}
\theoremstyle{plain}
\newtheorem{thm}{\protect\theoremname}[section]
\theoremstyle{plain}
\newtheorem{assumption}[thm]{\protect\assumptionname}
\theoremstyle{remark}
\newtheorem{rem}[thm]{\protect\remarkname}
\theoremstyle{plain}
\newtheorem{lem}[thm]{\protect\lemmaname}
\theoremstyle{plain}
\newtheorem{prop}[thm]{\protect\propositionname}
\theoremstyle{plain}
\newtheorem{cor}[thm]{\protect\corollaryname}

\providecommand{\assumptionname}{Assumption}
\providecommand{\corollaryname}{Corollary}
\providecommand{\lemmaname}{Lemma}
\providecommand{\propositionname}{Proposition}
\providecommand{\remarkname}{Remark}
\providecommand{\theoremname}{Theorem}

\begin{document}
\global\long\def\supp{\operatorname{supp}}%

\global\long\def\Uniform{\operatorname{Uniform}}%

\global\long\def\dif{\mathrm{d}}%

\global\long\def\e{\mathrm{e}}%

\global\long\def\ii{\mathrm{i}}%

\newcommand{\ssup}[1] {{\scriptscriptstyle{({#1}})}}

\global\long\def\Cov{\operatorname{Cov}}%

\global\long\def\Var{\operatorname{Var}}%

\global\long\def\pv{\operatorname{p.v.}}%

\global\long\def\e{\mathrm{e}}%

\global\long\def\p{\mathrm{p}}%

\global\long\def\R{\mathbf{R}}%

\global\long\def\Law{\operatorname{Law}}%

\global\long\def\supp{\operatorname{supp}}%

\global\long\def\image{\operatorname{image}}%

\global\long\def\dif{\mathrm{d}}%

\global\long\def\eps{\varepsilon}%

\global\long\def\sgn{\operatorname{sgn}}%

\global\long\def\tr{\operatorname{tr}}%

\global\long\def\Hess{\operatorname{Hess}}%

\global\long\def\Re{\operatorname{Re}}%

\global\long\def\Im{\operatorname{Im}}%

\global\long\def\Dif{\operatorname{D}}%

\global\long\def\divg{\operatorname{div}}%

\global\long\def\arsinh{\operatorname{arsinh}}%

\global\long\def\sech{\operatorname{sech}}%

\global\long\def\erf{\operatorname{erf}}%

\global\long\def\Cauchy{\operatorname{Cauchy}}%

\title{Additive-multiplicative stochastic heat equations, 
stationary solutions, and Cauchy statistics}
\author{Alexander Dunlap\and Chiranjib Mukherjee}
\thanks{A.D. was partially supported by an NSF Mathematical Sciences Postdoctoral Research Fellowship under Grant No.\ DMS-2002118 (held at NYU Courant). Research of C.M. is supported by the Deutsche Forschungsgemeinschaft (DFG) under Germany's Excellence Strategy EXC 2044-390685587, Mathematics M\"unster: Dynamics-Geometry-Structure.}
\begin{abstract}
    We study long-term behavior and stationary distributions for stochastic heat equations forced simultaneously by a multiplicative noise and an independent additive noise with the same distribution. We prove that nontrivial space-time translation-invariant measures exist for all values of the parameters. We also show that if the multiplicative noise is sufficiently strong, the invariant measure has Cauchy-distributed marginals. Using the same techniques, we prove a similar result on Cauchy-distributed marginals for a logarithmically attenuated version of the problem in two spatial dimensions. The proofs rely on stochastic analysis and elementary potential theory.
\end{abstract}

\address{(A.D.) Department of Mathematics, Duke University,
120 Science Dr., Durham, NC 27708,
USA.}
\email{alexander.dunlap@duke.edu}
\address{(C.M.) Fachbereich Mathematik und Informatik, Universität Münster, Einsteinstr. 62, Münster 48149, Germany.}
\email{chiranjib.mukherjee@uni-muenster.de}
\date{\today}
\maketitle

\section{Introduction and main results}

\subsection{Introduction} Let $(\Omega,\mathscr{H},\mathbb{P})$ be a probability space. Let
$d\in\mathbb{N}$, and let $\mathcal{X}$ be either $\mathbb{R}^{d}$
or $(\mathbb{R}/\mathbb{Z})^{d}$. Let $(\dif U_{t}(x))$ and $(\dif V_{t}(x))$
be iid homogeneous centered Gaussian noises on $\mathcal{X}$, with the
(formal) covariance kernels
\begin{equation}
    \mathbb{E}[\dif U_{t}(x)\dif U_{t'}(x')]=\mathbb{E}[\dif V_{t}(x)\dif V_{t'}(x')]=\delta(t-t')R(x-x'),\quad t,t'\in\mathbb{R},x,x'\in\mathcal{X}\label{eq:UVQV}
\end{equation}
and
\begin{equation}
    \mathbb{E}[\dif U_{t}(x)\dif V_{t'}(x')]=0,\qquad t,t'\in\mathbb{R},x,x'\in\mathcal{X}.\label{eq:uncorrelated}
\end{equation}
Here $R$ is a positive-definite covariance kernel. Precisely, the noises can be obtained by spatial convolution of a space-time white noise with a mollifier $\phi$ such that $\phi*\phi=R$. (See e.g.\ \cite{Wal86,Kho14} for preliminaries on space-time white noise.) We let $\mathscr{F}$
and $\mathscr{G}$ denote the $\sigma$-algebras generated by $(\dif U_{t}(x))$
and by $(\dif V_{t}(x))$, respectively. Throughout the paper, we assume that the following assumption on $\phi$ is in force (but see \cref{rem:assumption} below):
\begin{assumption}\label{assu:R}
We assume that either $\phi$ (and hence $R$) is smooth and compactly supported and that $\phi(x)\ge 0$ for all $x\in\mathcal{X}$, or
else that $d=1$ and $\phi=\delta$ (in which case $R=\delta$ as well and so $(\dif U_{t})$ and
$(\dif V_{t})$ are space-time white noises).
\end{assumption}

We fix $\alpha,\beta\ge0$ and consider the stochastic heat equation
with a combination of additive and multiplicative noise, which we term the \emph{additive-multiplicative stochastic heat equation} (AMSHE):
\begin{align}
\dif v_{t}(x) & =\frac{1}{2}\Delta v_{t}(x)\dif t+\beta v_{t}(x)\dif U_{t}(x)+\alpha\dif V_{t}(x),\qquad t\in\mathbb{R},x\in\mathcal{X}.\label{eq:theSPDE}
\end{align}
If $\beta=0$, then this reduces to the additive stochastic heat equation,
with Gaussian solutions \cite[Chapter 3]{Kho14}. If $\alpha=0$, then this reduces to the
multiplicative stochastic heat equation (MSHE, also called the parabolic
Anderson model), which is solved by the partition function of an associated
directed polymer \cite{Com17}.

In the present note we will be interested in the
long-time behavior of the problem, primarily in the new case $\alpha,\beta>0$. We summarize our most important results as follows:
\begin{itemize}
    \item We will show (\cref{thm:convergence} below) that the problem admits (statistically) space-time stationary solutions in any dimension and for any choice of $\alpha\ge 0$ and $\beta>0$. The space-time stationary solutions we construct arise as the long-time limits of the equation started with zero initial data.
    \item We will show (\cref{thm:stsproperties}(\ref{enu:strongcase}) below) that, whenever 
        $\beta$ and $\mathcal{X}$ are such that the only space-time stationary solution for the corresponding problem with $\alpha=0$ is zero (known as the strong-noise regime), then the space-time stationary solution described above yields a Cauchy-distributed random variable when tested against any positive measure. By contrast, in the bulk of the complementary weak-noise regime, the marginals of the space-time stationary solution admit finite $p$th moments for some $p>1$; see \cref{thm:stsproperties}(\ref{enu:Lpcase}) below.
        \item In the case $d=2$, $\mathcal{X}=\mathbb{R}^2$, in \cref{thm:2dthm}         
        we furthermore prove that Cauchy-distributed one-point statistics arise when $R$ is taken to be an approximation of $\delta$ at scale $\eps>0$ and $\alpha$ and $\beta$ are each attenuated by a factor of $\sqrt{\log\eps^{-1}}$, in the regime when $\beta\ge \sqrt{2\pi}$. This is the critical scaling and critical temperature studied in the case $\alpha =0$ in \cite{BC98,CSZ17,CSZ19,GQT19,CSZ21}.  In $d=2$, any positive $\beta$ is in the strong noise regime, and so this scaling should be thought of as ``blowing up'' the phase transition at $\beta =0$.
    \end{itemize}

    We find the appearance of the exact Cauchy distribution rather surprising, especially as it arises in the setting of invariant measures for infinite-dimensional stochastic evolution equations. We emphasize that \emph{no rescaling} is necessary to obtain the Cauchy statistics in the strong-noise regime. We are thus in the unusual position to be able to compute exact non-Gaussian marginal distributions of the invariant measure of a stochastic PDE, and moreover one in which the noise covariance $R$ is not required to be of a particular form (e.g.\ $R=\delta$) to produce an integrable structure. We also emphasize that we do \emph{not} expect multipoint distributions of the spacetime-stationary solutions to be jointly Cauchy, although we prove that \emph{positive} linear combinations of the values of the spacetime-stationary solution at a single time are Cauchy; see \cref{rem:samesign} below.

We see the AMSHE as a simple case of a stochastic heat equation with some multiplicative structure, but in which the solutions are not required to remain positive. 
We expect that nonlinear stochastic heat equations, of the form e.g.\ \cref{eq:theSPDE-2} below, admit similar invariant measures. However, except perhaps for special choices of the nonlinearity and/or noise covariance kernel $R$, we believe that the appearance of exact Cauchy distributions for such nonlinear problems is unlikely. See \cref{rem-nonlinear} below. This work therefore sheds light on the respective roles of the behavior of the nonlinearity at $0$ and at infinity.

\subsection{Main results} Now we precisely state our main results.
We need to introduce some function spaces. For
a non-negative ``weight'' function $w\in\mathcal{C}(\mathcal{X};\mathbb{R}_{\ge0})$
and an exponent $p\in[1,\infty)$, we define the weighted $L^{p}$
space $L_{w}^{p}(\mathcal{X})$ with the norm
\begin{equation}
\|z\|_{L_{w}^{p}(\mathcal{X})}^{p}=\int_{\mathcal{X}}|z(x)|^{p}w(x)^{p}\,\dif x.\label{eq:weightednormdef}
\end{equation}
Also, for $\xi>0$, define the weight
\begin{equation}
    w_{\mathcal{X};\xi}(x)=\begin{cases}
1, & \mathcal{X}=(\mathbb{R}/\mathbb{Z})^{d};\\
\e^{-\xi|x|}, & \mathcal{X}=\mathbb{R}^{d}.
\end{cases}\label{eq:wdef}
\end{equation}
Our first main result concerns the long-time behavior of solutions
to \cref{eq:theSPDE}. Let $v^{\ssup T}$ solve \cref{eq:theSPDE}
with initial condition $0$ at time $-T$:
\begin{equation}
v_{-T}^{\ssup T}(x)=0\qquad\text{for all }x\in\mathcal{X}.\label{eq:vTic-1}
\end{equation}

\begin{thm}
\label{thm:convergence}%
For any $\alpha\ge0$ and any $\beta>0$,
there is a space-time stationary solution $(\tilde{v}_{t}^{\ssup{\alpha,\beta}})_t$
to \cref{eq:theSPDE}, with $\tilde{v}_{t}^{\ssup{\alpha,\beta}}\in\bigcap_{p\in [1,\infty),\xi>0} L^p_{w_{\mathcal{X};\xi}}(\mathcal{X})$ a.s. for each $t$, such that for each $t\in\mathbb{R}$, $\xi>0$, and $p\in[1,\infty)$, we have
\[
    \lim_{T\to\infty}\|v_{t}^{\ssup T}-\tilde{v}_{t}^{\ssup{\alpha,\beta}}\|_{L_{w_{\mathcal{X};\xi}}^{p}(\mathcal{X})}=0\qquad\text{in probability.}
\]
\end{thm}

When we say that $(\tilde v^{\ssup{\alpha,\beta}}_t)$ is a space-time stationary solution to \cref{eq:theSPDE}, we mean that it solves \cref{eq:theSPDE} and that it is invariant in law under translations in both space and time.

We note that \cref{thm:convergence} is false in general for $\beta = 0$ (the additive stochastic heat equation). In this case, space-time stationary solutions exist when $\mathcal{X}=\mathbb{R}^{d}$ for $d\ge 3$, but not otherwise. In particular, at a fixed time these space-time stationary solutions have the distribution of a smoothed Gaussian free field (GFF), and the GFF can be defined on the whole space only when $d\ge 3$. On the torus, the additive stochastic heat equation has a zero-frequency mode undergoing a Brownian motion, and thus also does not admit an invariant measure in any dimension. The spatial gradient of the GFF can be defined on the whole space in any dimension, however, and indeed space-time stationary solutions for the gradient of the solution to the additive SHE do always exist. %

It is also known that when $\mathcal{X}=\mathbb{R}^{d}$ for $d\ge3$,
there is a phase transition, from so-called weak noise (or weak disorder) to strong noise (disorder), for the MSHE, which
is the case of \cref{eq:theSPDE} with $\beta>0$ and $\alpha=0$. We will exhibit a related phase transition
when $\alpha,\beta>0$. First we recall the parameters associated with
the phase transition. We denote by $(u_{t}^{\ssup T})$ the process solving
the MSHE
\begin{align}
    \dif u_{t}^{\ssup T}(x) & =\frac{1}{2}\Delta u_{t}^{\ssup T}(x)\dif t+\beta u_{t}^{\ssup T}(x)\dif U_{t}(x), &  & t>-T,x\in\mathcal{X};\label{eq:duT}\\
u_{-T}^{\ssup T}(x) & =1, &  & x\in\mathcal{X}.\label{eq:uTic}
\end{align}
It can be shown (see \cref{prop:martingaleQVs} below), and is indeed well-known in the literature, that the process $(u_{0}^{\ssup T}(0))_{T\ge0}$ is a
nonnegative martingale, and so the limit
\begin{equation}
    M_{\infty}\coloneqq\lim_{T\to\infty}u_{0}^{\ssup T}(0)\label{eq:Zinfty}
\end{equation}
exists almost surely. This is the reason for considering the equation started at time $-T$ rather than at time $0$: the solution of the equation started at time $0$ and evaluated at time $T$ would not have this martingale structure as $T$ is varied.

The following proposition summarizes known results about the phase transition; see \cref{sec:phase-transition} for detailed references.
\begin{prop}\label{prop:phase-transition}
There are two disjoint intervals $I_{\mathrm{weak}}=I_{\mathrm{weak}}(\mathcal{X},R)$
and $I_{\mathrm{strong}}=I_{\mathrm{strong}}(\mathcal{X},R)$ such
that $I_{\mathrm{weak}}\cup I_{\mathrm{strong}}=[0,\infty)$ and the
following properties hold:
\begin{enumerate}
    \item\label{enu:weak} If $\beta\in I_{\mathrm{weak}}$, then the martingale $(u_{0}^{\ssup T}(0))_{T\ge0}$
        is uniformly integrable and $M_{\infty}$ is nonzero almost surely.
In this case we say that \emph{weak disorder} holds.
\item\label{enu:strong} If $\beta\in I_{\mathrm{strong}}$, then $M_{\infty}=0$ almost surely.
In this case we say that \emph{strong disorder} holds.
\end{enumerate}
Define $\beta_{\mathrm{c}}\coloneqq\sup I_{\mathrm{weak}}=\inf I_{\mathrm{strong}}$. Then we additionally have
\begin{enumerate}[resume]
    \item\label{enu:dge3} If $\mathcal{X}=\mathbb{R}^{d}$ with $d\ge3$, then $\beta_{\mathrm{c}}>0$.
    \item\label{enu:dle2} If $\mathcal{X}$ is $\mathbb{R}^{d}$ for $d\le 2$
or $(\mathbb{R}/\mathbb{Z})^{d}$ for any $d$, then $\beta_{\mathrm{c}}=0$.
\item\label{enu:ILprelation} Let \begin{equation}I_{L^p} = I_{L^p}(\mathcal{X},R) \coloneqq \{\beta\ge 0\ :\ \text{the martingale }(u_0^{\ssup T})_T\text{ is bounded in }L^p(\Omega)\}.\label{eq:ILpdef}\end{equation} Then we have %
\begin{equation}\label{eq:ILprelation}
[0,\beta_{\mathrm{c}})\subseteq\bigcup_{p\in(1,\infty)}I_{L^{p}}\subseteq [0,\beta_{\mathrm{c}}].
\end{equation}
\end{enumerate}
\end{prop}
\begin{rem}
The behavior of the problem at $\beta=\beta_{\mathrm{c}}$ has not been studied. Very recent results %
\cite{JL24a,JL24b} in the discrete setting suggest that %
the second inclusion in \cref{eq:ILprelation} is in fact an identity. %
\end{rem}

The phase transition associated with the MSHE
\cref{eq:duT} has implications for the invariant measures for the AMSHE as well. This is the content
of our second main theorem.
\begin{thm}
\label{thm:stsproperties}Let $\alpha\ge0$ and $\beta>0$. Depending
on the value of $\beta$, we make the following claims:
\begin{enumerate}
\item \label{enu:Lpcase}If $p\in(1,\infty)$ and $\beta\in I_{L^{p}}$,
then for each $x\in\mathcal{X}$ we have $\tilde{v}_{0}^{\ssup{\alpha,\beta}}(x)\in L^{p}(\Omega)$
and
\begin{equation}
\lim_{T\to\infty}\mathbb{E}|v_{0}^{\ssup T}(x)-\tilde{v}_{0}^{\ssup{\alpha,\beta}}(x)|^{p}=0.\label{eq:Lpstatement}
\end{equation}
Hence, in particular,
\[
\mathbb{E}|\tilde{v}_{0}^{\ssup{\alpha,\beta}}(x)|^{p}<\infty\qquad\text{and}\qquad\mathbb{E}[\tilde{v}_{0}^{\ssup{\alpha,\beta}}(x)]=0.
\]
    \item \label{enu:strongcase}If $\beta\in I_{\mathrm{strong}}$, then whenever $x_1,\ldots,x_n\in \mathcal{X}$ and $\gamma_1,\ldots,\gamma_n\ge0$ are such that $\sum_{i=1}^n \gamma_i = 1$, then %
\begin{equation}
    \sum_{i=1}^n \gamma_i\tilde{v}_{0}^{\ssup{\alpha,\beta}}(x_i)\sim\Cauchy(\alpha/\beta).\label{eq:Cauchystatement}
\end{equation}
Hence, in particular, $\mathbb{E}|\tilde v_0^{\ssup{\alpha,\beta}}(x)|=\infty$ for each $x\in\mathcal{X}$.
\end{enumerate}
\end{thm}

Here, we use $\Cauchy(c)$ to refer to the distribution of $c$ times a standard Cauchy random variable. So, for $c>0$, the pdf of the $\Cauchy(c)$ distribution is
\[
x\mapsto \frac 1{\pi c(1+(x/c)^2)},
\]
and the degenerate $\Cauchy(0)$ distribution is simply a $\delta$ distribution at the origin.

\begin{rem}\label{rem:samesign}
    The hypothesis that $\gamma_i\ge 0$ for each $i$ 
is essential for our proof of \cref{eq:Cauchystatement} to hold,
and we do not expect \cref{eq:Cauchystatement} to hold (even with a different parameter for the Cauchy distribution) if
the $\gamma_i$s are not all of the same sign. In fact, we do not have a conjecture for the distribution of the left side of \cref{eq:Cauchystatement} if the $\gamma_i$s have different signs.

    Families of random variables such that
all positive linear combinations are Cauchy distributed, but linear
combinations with mixed signs are in general not, have been recently
observed in several works in the statistics literature \cite{DX16,PM16,XCDS22}.
The complete scope of this phenomenon seems to not yet be well-understood. Moreover, these works also do not have any results or conjectures for the distribution of linear combinations with mixed signs. It is an intriguing question whether the statistics of the mixed-sign linear combinations enjoy any universality properties. If so, the present work gives another example that is likely to be in the same universality class and can be studied using martingale techniques.
\end{rem}

\begin{rem}\label{rem-nonlinear}
The equation \cref{eq:theSPDE} is similar to the problem
\begin{align}
    \dif\overline{v}_{t}(x) & =\frac{1}{2}\Delta\overline{v}_{t}(x)\dif t+(\beta^{2}\overline{v}_{t}(x)^{2}+\alpha^{2})^{1/2}\dif U_{t}(x);\qquad t\in\mathbb{R},x\in\mathcal{X}.\label{eq:theSPDE-2}
\end{align}
In fact, in the case when $R=\delta$ (space-time white noise), the
noises can be coupled such that \cref{eq:theSPDE} and \cref{eq:theSPDE-2}
have the same solutions. The equation \cref{eq:theSPDE-2} is
an example of a \emph{nonlinear stochastic heat equation}. When $d\ge3$, 
the problem \cref{eq:theSPDE-2} with $\alpha,\beta$ very small was shown by Gu and Li in \cite{GL19} to admit nontrivial spacetime-stationary solutions, as well as Edwards--Wilkinson
fluctuations after a diffusive rescaling. (See \cite{GRZ18,DGRZ20,CCM19} for Edwards--Wilkinson fluctuations in the $\alpha=0$ case.) It seems likely that Edwards--Wilkinson fluctuations hold in the subcritical regime for the present model as well. On the other hand, it is also likely that the behavior of \cref{eq:theSPDE-2}, and even that of equations with more generality nonlinearities having similar behavior at $0$ and $\pm\infty$, is qualitatively similar to those in the present study, although the appearance of exact Cauchy distributions seems unlikely. We leave these questions for future work.
\end{rem}

\subsubsection*{The logarithmically attenuated two-dimensional problem}
We also consider a logarithmically attenuated problem in $d=2$ (i.e. $\mathcal{X} = \mathbb{R}^2$, although the following discussion applies just as well to $\mathcal{X} = \mathbb{T}^2$). Consider the problem
\begin{equation}\label{eq:logproblem}
    \dif v_t^\eps = \frac12\Delta v_t^\eps (x)\dif t + \frac{\beta}{\sqrt{\log\eps^{-1}}}v_t^\eps(x)\dif U_t^\eps(x) + \frac\alpha{\sqrt{\log\eps^{-1}}}\dif V_t^\eps(x).
\end{equation}
Here, $U_t^\eps$ and $V_t^\eps$ are iid Gaussian noises that are white in time and have correlation kernel \[\mathbb{E} U_t^\eps(x)U_{t'}^\eps(x') = \mathbb{E} V_t^\eps(x)V_{t'}^\eps(x')=\delta(t-t')R^\eps(x-x'),\] where $R^\eps(y) = \eps^{-2}R(\eps^{-1}y)$ for some smooth, positive, compactly-supported correlation function $R$.

This problem with $\alpha=0$ was previously studied in \cite{CSZ17}.
In \cite[Theorem~2.15]{CSZ17}, it was shown that if $\alpha=0$, then for any fixed $t>0$ and $x\in\mathbb{R}^2$, we have 
\begin{equation}\label{eq:2dphasetrans}
    v^\eps_t(x) \xrightarrow[\eps\searrow 0]{\mathrm{law}}\begin{cases} \exp\{Z_\beta-\frac12\Var Z_\beta\}&\text{if }\beta<\sqrt{2\pi};\\
    0&\text{if }\beta\ge\sqrt{2\pi},\end{cases}
    \end{equation} where $Z_\beta\sim N(0,\log\frac1{1-\beta^2/(2\pi)})$. That is, there is a phase transition at $\beta=\sqrt{2\pi}$. We will prove the following generalization to $\alpha\ge0$ of the result \cref{eq:2dphasetrans} in the supercritical case.
\begin{thm}\label{thm:2dthm}
    If $\beta\ge\sqrt{2\pi}$ and $\alpha \ge0$, then for any  $x_1,\ldots,x_n\in\mathbb{R}^2$ and
    $\gamma_1,\ldots,\gamma_n\ge 0$ with $\sum_{i=1}^n \gamma_i=1$, we have \begin{equation}\lim\limits_{\eps\searrow 0}\Law\left(\sum_{i=1}^n\gamma_iv_t^\eps(x_i)\right)=\Cauchy(\alpha/\beta).\label{eq:2dconcl}\end{equation}
\end{thm}
\cref{thm:2dthm} is consistent with a computation of \cite{DG23a}, in which, for a related model, a statement analogous to %
\begin{equation}\label{eq:2dsubcritcase}
\adjustlimits\lim_{\beta\nearrow\sqrt{2\pi}}\lim_{\eps\searrow 0} \Law\left(\sum_{i=1}^n \gamma_i v_t^\eps(x_i)\right) = \Cauchy(\alpha/\sqrt{2\pi}).
\end{equation}
is shown.
(In fact, \cref{eq:2dsubcritcase} does not require the hypothesis that $\gamma_i\ge 0$.) We can thus interpret \cref{thm:2dthm} as in particular stating that the limits in \cref{eq:2dsubcritcase} can be exchanged.

The problem \cref{eq:logproblem} with $\alpha=0$ and $\beta=\sqrt{2\pi}$ is known to converge as a \emph{process} in $(t,x)$ to the so-called \emph{critical stochastic heat flow} \cite{CSZ21}. We leave the process-level behavior of \cref{eq:logproblem} with $\alpha> 0$ as a problem for future study.

\subsection{Source of the Cauchy distribution} Here we give a sketch of the origin of the Cauchy distribution in \cref{thm:stsproperties}(\ref{enu:strongcase}). The argument proceeds by a martingale analysis. We will study two martingales, corresponding to the MSHE and the AMSHE as the starting time $T$ is moved backward in time. By the Duhamel formula for the AMSHE and Itô's formula, we will show that these martingales satisfy the hypotheses of the following key observation:
\begin{prop}\label{prop:howtogetCauchy}
    Suppose that $(M_t)_{t\ge 0}$ and $(N_t)_{t\ge 0}$ are continuous martingales (with respect to the same filtration) satisfying the following hypotheses almost surely:
    \begin{enumerate}
        \item $[M]_t = [N]_t$ for all $t\ge 0$.\label{enu:sameqv}
        \item $[M,N]_t = 0$ for all $t\ge 0$.\label{enu:nojointqv}
        \item $M_t \ge 0$ for all $t\ge 0$.\label{enu:positive}
        \item $\lim\limits_{t\to\infty} M_t  = 0$.\label{enu:tozero}
        \item $N_0 =0$.\label{enu:N0iszero}
    \end{enumerate}
    Then $N_t$ converges almost surely to a $\Cauchy(M_0)$-distributed limit as $t\to\infty$.
\end{prop}
The statement and proof of \cref{prop:howtogetCauchy} will be split across several lemmas in the sequel, since we will need some of the intermediate results for other parts of our argument. For clarity, we give a compact proof here.
\begin{proof}[Proof of \cref{prop:howtogetCauchy}]
    By the Dubins--Schwarz Theorem and hypotheses~(\ref{enu:sameqv}) and~(\ref{enu:nojointqv}), there is a random variable $\sigma$ and nondecreasing time change $g\colon [0,\sigma)\to [0,\infty)$ with $g(0)=0$ and $\lim\limits_{s\nearrow\sigma} g(s)=\infty$ such that $(M_{g(s)},N_{g(s)})_{s\in [0,\sigma)}$ is a standard two-dimensional Brownian motion. If we extend this Brownian motion to a Brownian motion $(X_s,Y_s)$ on $[0,\infty)$, Hypotheses~(\ref{enu:positive}) and~(\ref{enu:tozero}) then imply that $\sigma$ is the first hitting time of the Brownian motion $(X_s)_s$ to zero. This means that $\lim\limits_{t\to\infty} N_t = \lim\limits_{s\to\sigma} Y_s = Y_\sigma$ is the second coordinate of the exit location of a standard two-dimensional Brownian motion, started at $(M_0,0)$, from the right half plane. The distribution of this object is given by the Poisson kernel for the right half-plane, which is well known to be the pdf of a $\Cauchy(M_0)$ random variable. (See \cite[§2.2.4b]{Eva10} or \cite[Lemma~3]{Spi58}).
    \end{proof}

    \begin{rem}\label{rem:assumption}
The simplicity of the proof of \cref{prop:howtogetCauchy} means that we expect our results to hold in rather greater generality. For example, \cref{assu:R}  imposed on the covariance function $R$ is not optimal, and we expect the main results of the paper to be true for many non-smooth and non-compactly-supported correlation kernels as well. We also expect similar results to hold for more general spaces $\mathcal{X}$, e.g.\ in the setting of continuous-time directed polymers on discrete lattices or even more general graphs. We have chosen the rather restrictive assumptions to keep the amount of technical baggage in the paper to a minimum and clearly elucidate the key ideas.
\end{rem}

\begin{rem}\label{rem:need-cts} On the other hand, a key assumption in \cref{prop:howtogetCauchy} is that the martingales are continuous. This is essential for the application of the Dubins--Schwarz theorem, and the proposition is not true without this assumption. Therefore, we do not expect Cauchy statistics to arise for the invariant measures of similar discrete-time problems, although we do expect such invariant measures to have qualitatively similar features.
\end{rem}

\subsection*{Organization of the paper}
In \cref{sec:mult-duhamel}, we introduce the Duhamel formula for solutions to \cref{eq:theSPDE}, which forms the backbone of our analysis. In \cref{sec:onept}, we show that the one-point distributions of solutions converge. In \cref{sec:cauchy}, we show that in the supercritical regime the limiting statistics are Cauchy, and also prove \cref{thm:2dthm}. In \cref{sec:stationary}, we show that spacetime-stationary solutions exist at the process level, and prove \cref{thm:convergence,thm:stsproperties}.

\subsection*{Acknowledgments}
We thank Nikos Zygouras for interesting discussions. We also thank two anonymous referees for careful readings of the manuscript and many important comments.

\section{The multiplicative problem and the Duhamel formula}\label{sec:mult-duhamel}

It will often suit our purposes to represent solutions to \cref{eq:theSPDE}
in terms of a Duhamel formula. To this end, we let ${Z}_{s,t}^{\ssup\beta}(y,x)$
denote the Green's function for the multiplicative problem \cref{eq:duT},
solving
\begin{align}
\dif_{t}{Z}_{s,t}^{\ssup\beta}(y,x) & =\frac{1}{2}\Delta_{x}{Z}_{s,t}^{\ssup\beta}(y,x)\dif t+\beta{Z}_{s,t}^{\ssup\beta}(y,x)\dif U_{t}(x), &  & -\infty<s<t<\infty\text{ and }x,y\in\mathcal{X};\label{eq:Zeqn}\\
{Z}_{s,s}^{\ssup\beta}(y,x) & =\delta(x-y), &  & s\in\mathbb{R}\text{ and }x,y\in\mathcal{X}.\label{eq:Zicdelta}
\end{align}
If $R$ is spatially smooth, then we have the Feynman--Kac formula
\begin{equation}
    {Z}_{s,t}^{\ssup\beta}(y,x)=G_{t-s}(x-y)\mathrm{E}_{X}^{(s,y)\rightsquigarrow(t,x)}\bigg[\exp\left\{ \int_{s}^{t}\dif U_{r}(X(r))-\frac{1}{2}\beta^{2}R(0)(t-s)\right\} \bigg],\label{eq:feynmankac}
\end{equation}
where \[G_t(x) = \begin{cases}
(2\pi t)^{-d/2} \e^{-|x|^2/(2t)},&\mathcal{X}=\mathbb{R}^d;\\
\sum_{y\in\mathbb{Z}^d}(2\pi t)^{-d/2} \e^{-|x+y|^2/(2t)}, &\mathcal{X}=\mathbb{T}^d\end{cases}\] is the standard heat kernel, and
$\mathrm{E}_{X}^{(s,y)\rightsquigarrow(t,x)}$ denotes expectation
with respect to a unit-diffusivity Brownian bridge $X$ such that $X(s)=y$ and $X(t)=x$.
If $R=\delta$ (and hence $d=1$), then the expression \cref{eq:feynmankac} becomes
singular, but the solution to \crefrange{eq:Zeqn}{eq:Zicdelta} can nonetheless be constructed; see \cite{AKQ14b,AJRAS22a}. (These references deal with the case $\mathcal{X}=\mathbb{R}$, but the construction in the case $\mathcal{X}=\mathbb{T}$ is simply the same but with periodized noise and  so poses no additional difficulty.)

The following basic property of the propagator will often be useful:
\begin{lem}\label{lem:momentbd}
    For each $s<t$ and $x\in\mathcal{X}$, the process $y\mapsto G_{t-s}^{-1}(x-y){Z}^{\ssup\beta}_{s,t}(y,x)$ is stationary in $y$. Moreover, we have for each fixed $s\in\mathbb{R}$, $T\in(0,\infty)$, and $p\in [1,\infty)$ that
    \begin{equation}
        \adjustlimits\sup_{t\in[s,s+T]}\sup_{y\in\mathcal{X}} \mathbb{E}[(G_{t-s}(x-y)^{-1}{Z}^{\ssup\beta}_{s,t}(y,x))^p]<\infty.\label{eq:moment-bd}
    \end{equation}
\end{lem}

In the case $R=\delta$ and $\mathcal{X}=\mathbb{R}$, the stationarity in \cref{lem:momentbd} is a consequence of the symmetry properties proved in \cite[Proposition~2.3]{AJRAS22a}; see \cite[Corollary~2.4]{AJRAS22a} or \cite[Proposition~1.4]{ACQ11}. In this case, the moment bound have been proved as \cite[Lemma~3.2]{AJRAS22a} (which only considers the case $y=0$, but spatial stationarity implies that the supremum can be taken over $y$). The cases of periodic and smooth noise can be handled in an identical fashion, so we omit the proofs here.

By the linearity
of the equation \cref{eq:Zeqn}, we have a Chapman--Kolmogorov-type equation: almost surely,
whenever $s<r<t$,
\begin{equation}
{Z}_{s,t}^{\ssup\beta}(y,x)=\int_{\mathcal{X}}{Z}_{s,r}^{\ssup\beta}(y,z){Z}_{r,t}^{\ssup\beta}(z,x)\,\dif z.\label{eq:CKthing}
\end{equation}
See \cite[Lemma~3.12]{AJRAS22a} for the case $R=\delta$ and $\mathcal{X}=\mathbb{R}$; the proofs of the other cases are identical.

For $s<t$, we have the following backward-in-time Duhamel formula for ${Z}^{\ssup \beta}_{s,t}$:
\begin{equation}
    {Z}^{\ssup \beta}_{s,t}(y,x) = G_{t-s}(x-y) - \beta\int_s^t\int_{\mathcal{X}} G_{r-s}(y-z){Z}^{\ssup \beta}_{r,t}(z,x)\,\dif U_{r}(z).\label{eq:Zmild}
\end{equation}
In the case $R=\delta$ and $\mathcal{X}=\mathbb{R}$, this follows from forward-in-time Duhamel formula \cite[(2.2)]{AJRAS22a} and the time-reversal symmetry \cite[Proposition~2.3(ii)]{AJRAS22a}. The case when $R$ is smooth is completely analogous; %
alternatively, in the case when $R$ is smooth, \cref{eq:Zmild} is straightforward to derive from the Feynman--Kac formula \cref{eq:feynmankac} and the Itô formula.
Also, for $s<t$, we can write a Duhamel-type formula for the solution
to \cref{eq:theSPDE}:
\begin{equation}
    v_{t}(x)=\int_{\mathcal{X}}{Z}_{s,t}^{\ssup\beta}(y,x)v_{s}(y)\,\dif y+\alpha\int_{s}^{t}\int_{\mathcal{X}}{Z}_{r,t}^{\ssup\beta}(y,x)\,\dif V_r(y).\label{eq:Duhamel}
\end{equation}
The fact that the stochastic integral is well-defined is a straightforward consequence of the moment bound stated in \cref{lem:momentbd}.

\begin{lem}
For any fixed $\xi>\tilde{\xi}>0$, $s<t$, and $p\in[1,\infty)$, we have almost
surely that
\begin{equation}
    \int_{\mathcal{X}}\left(w_{\mathcal{X};\xi}(x)\sup_{z\in\mathcal{X}}\frac{{Z}_{s,t}^{\ssup\beta}(z,x)}{w_{\mathcal{X};\tilde\xi}}(z)\right)^{p}\,\dif x<\infty\label{eq:Zpinftybd}
\end{equation}
and, if $1/p+1/q=1$,
\begin{equation}
    \int_{\mathcal{X}}\left(w_{\mathcal{X};\xi}(x)\left(\int_{\mathcal{X}} \left(\frac{{Z}_{s,t}^{\ssup\beta}(z,x)}{w_{\mathcal{X};\tilde\xi}}(z)\right)^q\,\dif z\right)^{1/q}\right)^p\,\dif x<\infty.\label{eq:Zpqbd}
\end{equation}
\end{lem}

\begin{proof}
If $\mathcal{X}=(\mathbb{R}/\mathbb{Z})^{d}$, then this follows simply
from the continuity of ${Z}_{s,t}^{\ssup\beta}(z,x)$ in $z$
and $x$.

Now suppose that $\mathcal{X}=\mathbb{R}^{d}$. For any $\kappa>0$,
we have
\begin{equation}
    \sup_{z,x\in\mathcal{X}}\left[\e^{-\kappa|x|-\kappa|z|+|z-x|^{2}/(2(t-s))}{Z}_{s,t}^{\ssup\beta}(z,x)\right]<\infty,\qquad\text{a.s.}\label{eq:introXi}
\end{equation}
This is a consequence of \cite[(3.14)]{AJRAS22a} in the case $d=1$ and $R=\delta$.
Indeed, we have
\begin{align*}
  \sup_{x,z\in\mathbf{R}}& \left[\e^{-\kappa|x| -\kappa |z|}G_{t-s}(z-x)^{-1}Z_{s,t}^{\ssup \beta}(z,x)\right]\\
                         &=\sup_{K\ge 1}\left[\sup_{x,z\;:\;|x|\vee|z|\in [K-1,K]}\e^{-\kappa|x| -\kappa |z|}G_{t-s}(z-x)^{-1}Z_{s,t}^{\ssup \beta}(z,x)\right]\\
                         &\le C_\kappa\sup_{K\ge 1}\left[K^{-4}\sup_{x,z\;:\;|x|,|z|\le K}G_{t-s}(z-x)^{-1}Z_{s,t}^{\ssup \beta}(z,x)\right],
\end{align*} for an absolute constant $C_\kappa<\infty$ depending only on $\kappa$, and then the last quantity is almost surely finite by \cite[(3.14)]{AJRAS22a}. %
    In the other cases it can be proved similarly, using the spatial stationarity
of $G_{t-s}(z-x)^{-1}{Z}_{s,t}^{\ssup\beta}(z,x)$ and moment
bounds (see \cref{eq:moment-bd}) on the solution to the MSHE along with a union bound to control the quantity inside the supremum of \cref{eq:introXi} on $\mathbb{Z}^d$, and then using similar moment bounds on the derivative (which are easily obtained when $R$ is smooth) to extend this control to $\mathbb{R}^d$.

As a consequence of \cref{eq:introXi}, there is a random variable $\Xi$ such that, with
probability $1$,
\[
\e^{\tilde{\xi}|z|-\xi|x|}{Z}_{s,t}^{\ssup\beta}(z,x)\le\Xi\e^{(\kappa+\tilde{\xi})|z|+(\kappa-\xi)|x|-|z-x|^{2}/(2(t-s))}\qquad\text{for all }x,z\in\mathbb{R}^{d}.
\]
Now we choose $\kappa\in(0,\frac{\xi-\tilde{\xi}}{2})$, which
allows us to find a constant $\tilde{\kappa}>0$ such that
\[
(\kappa+\tilde{\xi})|z|+(\kappa-\xi)|x|-\frac{|z-x|^{2}}{2(t-s)}\le-\tilde{\kappa}(|x|+|z|)+\tilde{\kappa}^{-1}\qquad\text{for all }x,z\in\mathbb{R}^{d},
\]
so \cref{eq:introXi} becomes
\[
\e^{\tilde{\xi}|z|-\xi|x|}{Z}_{s,t}^{\ssup\beta}(z,x)\le\Xi\e^{-\tilde{\kappa}(|x|+|z|)+\tilde{\kappa}^{-1}}\qquad\text{for all }x,z\in\mathbb{R}^{d}.
\]
Then \cref{eq:Zpinftybd,eq:Zpqbd} follow easily.
\end{proof}

\section{Convergence of marginal distributions}\label{sec:onept}

In this section we show the convergence of the marginal distributions of the AMSHE solutions.

We consider the two initial value problems
\begin{align*}
\dif u_{t}^{\ssup T}(x) & =\frac{1}{2}\Delta u_{t}^{\ssup T}(x)\dif t+\beta u_{t}^{\ssup T}(x)\dif U_{t}(x),\qquad t>-T,x\in\mathcal{X};\\
u_{-T}^{\ssup T}(x) & =1
\end{align*}
and
\begin{align}
\dif v_{t}^{\ssup T}(x) & =\frac{1}{2}\Delta v_{t}^{\ssup T}(x)\dif t+\beta v_{t}^{\ssup T}(x)\dif U_{t}(x)+\alpha\dif V_{t}(x),\qquad t>-T,x\in\mathcal{X};\label{eq:vSPDE}\\
v_{-T}^{\ssup T}(x) & =0.\label{eq:vTic}
\end{align}
By the Duhamel formula \cref{eq:Duhamel}, we have for all $t\ge-T$
and $x\in\mathcal{X}$ that
\begin{equation}
u_{t}^{\ssup T}(x)=\int_{\mathcal{X}}{Z}_{-T,t}^{\ssup\beta}(y,x)\,\dif y\label{eq:uDuhamel}
\end{equation}
and
\begin{equation}
v_{t}^{\ssup T}(x)=\alpha\int_{-T}^{t}\int_{\mathcal{X}}{Z}_{r,t}^{\ssup\beta}(y,x)\,\dif V_{r}(y).\label{eq:vDuhamel}
\end{equation}
For $s<t$, we let $\mathscr{F}_{s,t}$ and $\mathscr{G}_{s,t}$ be
the $\sigma$-algebras generated by $(\dif U_{r})_{r\in[s,t]}$ and
$(\dif V_{r})_{r\in[s,t]}$, respectively. We let $\mathscr{H}_{s,t}=\mathscr{F}_{s,t}\vee\mathscr{G}_{s,t}$. We also recall that $\mathscr{F}$ is the $\sigma$-algebra generated by $(\dif U_r)_{r\in\mathbb{R}}$.

For a finite signed measure $\mu$ on $\mathcal{X}$ (i.e.\ one satisfying $|\mu|(\mathcal{X})<\infty$), we define
\begin{equation}
M_{T}(\mu)=\int_{\mathcal{X}}u_{0}^{\ssup T}(x)\,\dif\mu(x)\qquad\text{and}\qquad N_{T}(\mu)=\int_{\mathcal{X}}v_{0}^{\ssup T}(x)\,\dif\mu(x).\label{eq:Mdef}
\end{equation}
We note that $M_{T}(\delta_{x})=u_{0}^{\ssup T}(x)$, considered as a
process depending on $T$ for fixed $x$, is the martingale considered
in \cite{MSZ16}. (There it is called $Z_\eps$; see \cite[Remark~3.3]{MSZ16}.)
\begin{prop}\label{prop:martingaleQVs}
    For any finite signed measure $\mu$, the processes $(M_{T}(\mu))_{T\ge 0}$
and $(N_{T}(\mu))_{T\ge 0}$ are continuous $(\mathscr{H}_{-T,0})_{T}$-martingales.
Their quadratic variation processes satisfy
\begin{equation}
    \begin{aligned}\alpha^{2}&\dif[M(\mu)]_{T}\\& =\alpha^{2}\beta^{2}\iiiint_{\mathcal{X}^{4}}{Z}_{-T,0}^{\ssup\beta}(y_{1},x_{1}){Z}_{-T,0}^{\ssup\beta}(y_{2},x_{2})R(y_{1}-y_{2})\,\dif y_{1}\,\dif y_{2}\,\dif\mu(x_{1})\,\dif\mu(x_{2})\\
 & =\beta^{2}\dif[N(\mu)]_{T}
\end{aligned}
\label{eq:QVrelation}
\end{equation}
and
\begin{equation}
\dif[M(\mu),N(\mu)]_{T}=0.\label{eq:nojointQV}
\end{equation}
\end{prop}

\begin{proof}
    Let us first show that the process $(M_T(\mu))_{T\geq 0}$ is a martingale. Indeed, we have by \cref{eq:uDuhamel,eq:Mdef,eq:Zmild}
that
\begin{align}
    M_T(\mu) &= \iint_{\mathcal{X}^2} {Z}^{\ssup\beta}_{-T,0}(z,x)\,\dif z\,\dif \mu (x)\notag \\&= \iint_{\mathcal{X}^2} \left(G_{T}(x-y)-\beta \int_{-T}^0\int_{\mathcal{X}} G_{r+T}(z-y){Z}^{\ssup \beta}_{r,0}(y,x)\,\dif U_r(y)\right)\,\dif y\,\dif\mu(x)\notag
           \\& = \mu(\mathcal{X})- \int_{-T}^0\int_{\mathcal{X}}\left(\beta\int_{\mathcal{X}} {Z}^{\ssup \beta}_{r,0}(y,x)\,\dif \mu(x)\right)\,\dif U_r(y).\label{eq:MmuIto}
\end{align}
The last term is a stochastic integral and hence a continuous local martingale. We compute the differential quadratic variation %
    \begin{align}
        \frac{\dif [M(\mu)]_T}{\dif T} %
                        &= \beta^2\iiiint_{\mathcal{X}^4} R(y_1-y_2){Z}_{-T,0}^{\ssup \beta}(y_1,x_1){Z}_{-T,0}^{\ssup \beta}(y_2,x_2)\,\dif y_1 \,\dif y_2\,\dif \mu (x_1)\,\dif \mu(x_2),
\end{align}
and hence conclude the first identity in \cref{eq:QVrelation}. The moment bound \cref{lem:momentbd} implies that the process $(M_T(\mu))_T$ is indeed a martingale.

On the other hand, \cref{eq:vDuhamel} and \cref{eq:Mdef}
imply that
\begin{equation}
    N_{T}(\mu)=\alpha\int_{\mathcal{X}}\int_{-T}^{0}\int_{\mathcal{X}}{Z}_{r,0}^{\ssup\beta}(y,x)\,\dif V_{r}(y)\,\dif\mu(x),\label{eq:NDuhamel}
\end{equation}
whence Itô's formula implies that
\begin{equation}
\dif N_{T}(\mu)=-\alpha\iint_{\mathcal{X}^{2}}{Z}_{-T,0}^{\ssup\beta}(y,x)\,\dif V_{-T}(y)\,\dif\mu(x).\label{eq:NmuIto}
\end{equation}
Taking the quadratic variation using \cref{eq:UVQV}, we see
that
\begin{equation}
\frac{\dif[N(\mu)]_{T}}{\dif T}=\alpha^{2}\iiiint_{\mathcal{X}^{2}}{Z}_{-T,0}^{\ssup\beta}(y_{1},x_{1}){Z}_{-T,0}^{\ssup\beta}(y_{2},x_{2})R(x_{1}-x_{2})\,\dif y_{1}\,\dif y_{2}\,\dif\mu(x_{1})\,\dif\mu(x_{2}),\label{eq:NmuQV}
\end{equation}
and hence that the second identity in \cref{eq:QVrelation} holds. Standard regularity and moment bounds on ${Z}^{\ssup\beta}_{-T,0}(y,x)$ again imply that $(N_T(\mu))_T$ is a continuous martingale. Moreover,
\cref{eq:nojointQV} can be seen from \cref{eq:MmuIto,eq:NmuIto}
using the assumption \cref{eq:uncorrelated}.
\end{proof}
\begin{cor}\label{cor:dubinsschwarz}
    For any finite signed measure $\mu$ on $\mathcal{X}$, there is a filtration $\{\mathscr{K}_q\}_{q\ge 0}$ and
a standard two-dimensional Brownian motion $(W_{q}(\mu),X_{q}(\mu)))_{q\ge0}$,
adapted with respect to a filtration $\{\mathscr{K}_{q}\}_{q\ge0}$
and with $(W_{0}(\mu),X_{0}(\mu))=(\mu(\mathcal{X}),0)$, such that,
for each $T\ge0$, $[M(\mu)]_{T}$ is a $\mathscr{K}_{q}$-stopping
time and 
\begin{equation}
(M_{T}(\mu),(\beta/\alpha)N_{T}(\mu))=(W_{[M(\mu)]_{T}}(\mu),X_{[M(\mu)]_{T}}(\mu)).\label{eq:timechange}
\end{equation}
\end{cor}

\begin{proof}
This follows immediately from the Dubins--Schwarz theorem and \crefrange{eq:QVrelation}{eq:nojointQV}.
\end{proof}
\begin{prop}
    \label{prop:Mconverges}For any finite signed measure $\mu$ on $\mathcal{X}$,
the limit
\begin{equation}
M_{\infty}(\mu)\coloneqq\lim_{T\to\infty}M_{T}(\mu)\label{eq:Mconverges}
\end{equation}
exists almost surely. Also, we have
\begin{equation}
[M(\mu)]_{\infty}=\lim_{T\to\infty}[M(\mu)]_{T}<\infty\label{eq:QVconverges}
\end{equation}
almost surely, and
\begin{equation}
\lim\limits _{T\to\infty}N_{T}(\mu)=\alpha\beta^{-1}X_{[M(\mu)]_{\infty}}(\mu)\eqqcolon N_{\infty}(\mu)\label{eq:limitingdistn-1}
\end{equation}
almost surely.
\end{prop}

\begin{proof}
If $\mu$ is an unsigned measure, then $(M_{T}(\mu))_{T}$ is a nonnegative
(super)martingale, as can be seen directly from the definition \cref{eq:Mdef},
so it converges almost surely by the supermartingale convergence theorem.
In the general case of $\mu$ a signed measure, we can write $\mu=\mu^{+}-\mu^{-}$
for unsigned measures $\mu^{+}$ and $\mu^{-}$, and $M_{T}(\mu)=M_{T}(\mu^{+})-M_{T}(\mu^{-})$.
Since $M_{T}(\mu^{+})$ and $M_{T}(\mu^{-})$ both converge almost
surely, we see that $M_{T}(\mu)$ converges almost surely as well.

The fact that $M_{T}(\mu)$ converges almost surely immediately implies
\cref{eq:QVconverges}, and then \cref{eq:limitingdistn-1}
follows from \cref{eq:timechange}.
\end{proof}
\begin{lem}
    Let $p\in(1,\infty)$ and assume that $\beta\in I_{L^{p}}$. For any finite
signed measure $\mu$ on $\mathcal{X}$, we have 
\begin{equation}
\lim_{T\to\infty}\mathbb{E}\bigg[\big([M(\mu)]_{\infty}-[M(\mu)]_{T}\big)^{p/2}\bigg]=0.\label{eq:LpQVconvergence}
\end{equation}
\end{lem}

\begin{proof}
This is a general and standard fact about $L^{p}$-bounded martingales;
we include a proof for convenience. Since the martingale $(M_{T}(\mu))_{T\ge0}$
is bounded in $L^{p}(\Omega)$, we have by the Burkholder--Davis--Gundy
inequality a constant $C_{p}<\infty$ such that
\begin{equation}
    \mathbb{E}\big([M(\mu)]_{\infty}^{p/2}\big)\le C_{p}\mathbb{E}\bigg[\sup_{T\ge0}|M_{T}(\mu)-\mu(\mathcal{X})|^{p}\bigg].\label{eq:BDG}
\end{equation}
(Here we use the fact that $M_0(\mu) = \mu(\mathcal{X})$ so $(M_T(\mu)-\mu(\mathcal{X}))_T$ is a martingale starting at $0$.)
Also, by Doob's maximal inequality we have for all $\overline{T}>0$
that
\[
    \mathbb{E}\bigg[\sup_{T\in[0,\overline{T}]}|M_{T}(\mu)|^{p}\bigg]\le\left(\frac{p}{p-1}\right)^{p}\mathbb{E}\big[\left|M_{\overline{T}}(\mu)\right|^{p}\big].
\]
Since the right side is bounded by assumption, we obtain by the monotone
convergence theorem that
\begin{equation}
\mathbb{E}\bigg[\sup_{T\in[0,\infty)}|M_{T}(\mu)|^{p}\bigg]<\infty.\label{eq:maximal}
\end{equation}
Combining \cref{eq:BDG} and \cref{eq:maximal}, we see
that $\mathbb{E}([M(\mu)]_{\infty}^{p/2})<\infty$. Since, for each
$T$, we have $([M(\mu)]_{\infty}-[M(\mu)]_{T})^{p/2}\le[M(\mu)]_{\infty}^{p/2}$,
and we have just shown that right side is in $L^{1}(\Omega)$, the
claim \cref{eq:LpQVconvergence} follows from the dominated convergence
theorem.
\end{proof}

\section{Cauchy-distributed limits}\label{sec:cauchy}
In this section we continue from the arguments of the previous section to prove that Cauchy statistics arise from the AMSHE. We also give the proof of \cref{thm:2dthm} on the 2D case. The proofs complete the strategy indicated in the proof of \cref{prop:howtogetCauchy} in the introduction.
\begin{prop}
\label{prop:Cauchylimits}Suppose that $\mu$ is a finite \emph{unsigned}
measure (i.e.~$\mu(A)\ge0$ for all $A\subseteq\mathcal{X}$), and
that $M_{\infty}(\mu)=0$ almost surely. Then $N_{\infty}(\mu)\sim\Cauchy\left(\alpha\beta^{-1}\mu(\mathcal{X})\right)$.
\end{prop}

\begin{proof}
If $\mu(\mathcal{X})=0$ then the result is clear, so assume that
$\mu(\mathcal{X})>0$. This means that $(M_{T}(\mu))_{T}$ is a positive
martingale, so $W_{q}(\mu)>0$ for all $q<[M(\mu)]_{\infty}$. On
the other hand, since $M_{\infty}(\mu)=0$, we have $W_{[M(\mu)]_{\infty}}(\mu)=0$.
This means that 
\begin{equation}
[M(\mu)]_{\infty}=\iota(\mu)\coloneqq\inf\{q\ge0\ :\ W_{q}(\mu)=0\}.\label{eq:Misiotadef}
\end{equation}
Now $X_{\iota(\mu)}(\mu)$ is the vertical coordinate of the exit
time of a standard two-dimensional Brownian motion, started at $(\mu(\mathcal{X}),0)$,
from the right half plane. It is known from standard potential theory (or see \cite[§2.2.4b]{Eva10} or \cite[Lemma~3]{Spi58})
that \begin{equation}X_{\iota(\mu)}(\mu)\sim\Cauchy(\mu(\mathcal{X})).\label{eq:itscauchy}\end{equation}
Therefore, by \cref{eq:limitingdistn-1} and \cref{eq:Misiotadef},
we have
\[
N_{\infty}(\mu)=\alpha\beta^{-1}X_{[M(\mu)]_{\infty}}(\mu)=\alpha\beta^{-1}X_{\iota(\mu)}(\mu)\sim\Cauchy(\alpha\beta^{-1}\mu(\mathcal{X})).\qedhere
\]
\end{proof}
\begin{proof}[Proof of \cref{thm:2dthm}.]
    Define $\beta_\eps = (\log \eps^{-1})^{-1/2}\beta$ and $\alpha_\eps = (\log\eps^{-1})^{-1/2}\alpha$. 
    Let $\mu$ be a (finite) positive linear combination of delta functions.
    Define $M^\eps_{t}(\mu)$ and $N^\eps_{t}=N^\eps_{t}(\mu)$ analogously to \cref{eq:Mdef}, with $\alpha$ and $\beta$ replaced by $\alpha_\eps$ and $\beta_\eps$, respectively. %
    \cref{cor:dubinsschwarz} applies, and gives us a standard two-dimensional Brownian motion $(W_q^\eps(\mu),X_q^\eps(\mu))_{q\ge 0}$, with $(W_0,X_0) = (\mu(\mathcal{X}),0)$, such that for each $t\ge 0$, \[(M_{t}^\eps(\mu),(\beta/\alpha)N_{t}^\eps(\mu)) = (W^\eps_{[M^\eps(\mu)]_{t}}(\mu),X^\eps_{[M^\eps(\mu)]_{t}}(\mu)).\] Actually, since we have not speciifed any coupling between the processes for varying choices of $\eps$, we can assume a coupling such that $(W^\eps_q(\mu),X^\eps_q(\mu))_{q\ge 0}$ does not depend on $\eps$, and so we denote it simply by $(W_q(\mu),X_q(\mu))_{q\ge0}$.
    Similarly to as in the proof of \cref{prop:Cauchylimits}, we have $W_q(\mu)>0$ for all $q<[M^\eps(\mu)]_t$. By \cite[Theorem~2.15]{CSZ17}, we have $\lim\limits_{\eps\searrow 0} M_t^\eps(\mu) =0$ in probability, which means that $\lim\limits_{\eps\searrow 0} W_{[M^\eps(\mu)]_t}(\mu)=0$ in probability. By the continuity of Brownian motion, this implies that $\lim\limits_{\eps\searrow 0}[M^\eps(\mu)]_t=\iota(\mu)$ in probability, with $\iota(\mu)$ defined as in \cref{eq:Misiotadef}. Again using the continuity of Brownian motion, we see that this implies that \[(\beta/\alpha)N^\eps_t(\mu) = X_{[M^\eps(\mu)]_t}(\mu)\xrightarrow[\eps\searrow 0]{\mathrm{P}} X_{\iota(\mu)}\overset{\cref{eq:itscauchy}}\sim \Cauchy(\mu(\mathcal{X})).\]
    This implies \cref{eq:2dconcl}.
\end{proof}

\section{Construction of space-time stationary solutions}\label{sec:stationary}

In this section we construct the space-time stationary solutions and complete the proofs of  \cref{thm:convergence,thm:stsproperties}.

For $x\in\mathcal{X}$, define
\begin{equation}
\tilde{v}(x)=N_{\infty}(\delta_{x})\overset{\cref{eq:NDuhamel}}{=}\alpha\int_{-\infty}^{0}\int_{\mathcal{X}}{Z}_{r,0}^{\ssup\beta}(y,x)\,\dif V_{r}(y).\label{eq:vtildedef}
\end{equation}
From \cref{eq:limitingdistn-1} and Fubini's theorem, we have
with probability $1$ that
\[
\lim_{T\to\infty}v_{0}^{\ssup T}(x)=\tilde{v}(x)\text{ for a.e. }x\in\mathcal{X}.
\]
Indeed, we have 
\[
 \mathbb{E}
 \left(\int_{\mathcal{X}}\mathbf{1}\left\{\lim_{T\to\infty}v_{0}^{\ssup T}(x)\not =\tilde{v}(x)\right\} \,\dif x\right) = \int_{\mathcal{X}}\mathbb{P}\left(
\lim_{T\to\infty}v_{0}^{\ssup T}(x)\not =\tilde{v}(x)\right)\,\dif x\overset{\cref{eq:limitingdistn-1}}=0\]
 and so 
 \[\mathbb{P}\left(\int_{\mathcal{X}}\mathbf{1}\left\{\lim_{T\to\infty}v_{0}^{\ssup T}(x)\not =\tilde{v}(x)\right\}\,\dif x\ne 0\right) = 0.
 \]

\begin{prop}
\label{prop:Lpconvergence}Let $p\in[1,\infty)$. Recall the definition $w_{\mathcal{X};\xi}$
from \cref{eq:wdef}.
For any $\xi>0$, we have
\begin{equation}
\lim_{T\to\infty}\|\tilde{v}-v_{0}^{\ssup T}\|_{L_{w_{\mathcal{X};\xi}}^{p}(\mathcal{X})}=0\qquad\text{in probability.}\label{eq:convinprob}
\end{equation}
If we furthermore assume that $p\in(1,\infty)$ and that $\beta\in I_{L^{p}}$,
then we additionally have
\begin{equation}
\lim_{T\to\infty}\mathbb{E} \bigg[\|\tilde{v}-v_{0}^{\ssup T}\|_{L_{w_{\mathcal{X};\xi}}^{p}(\mathcal{X})}^{p}\bigg]=0.\label{eq:convinLp}
\end{equation}
\end{prop}

\begin{proof}
We first note that whenever $s<r<t$, we can write
\begin{equation}
    \begin{aligned}
    {Z}_{s,t}^{\ssup\beta}(y,x)\overset{\cref{eq:CKthing}}&{=}\int{Z}_{s,r}^{\ssup\beta}(y,z){Z}_{r,t}^{\ssup\beta}(z,x)\,\dif z\\&\le\left(\sup_{z\in\mathcal{X}}\frac{{Z}_{r,t}^{\ssup\beta}(z,x)}{w_{\mathcal{X};\xi/2}}(z)\right)\int w_{\mathcal{X};\xi/2}(z){Z}_{s,r}^{\ssup\beta}(y,z)\,\dif z\end{aligned}\label{eq:applyCK}
\end{equation}
by the $L^{1}$/$L^{\infty}$ Hölder inequality. Subtracting \cref{eq:vDuhamel}
from \cref{eq:vtildedef}, we see that
\[
\tilde{v}(x)-v_{0}^{\ssup T}(x)=\alpha\int_{-\infty}^{-T}\int_{\mathcal{X}}{Z}_{r,0}^{\ssup\beta}(y,x)\,\dif V_{r}(y).
\]
Therefore, we have
\[
\|\tilde{v}-v_{0}^{\ssup T}\|_{L_{w_{\mathcal{X};\xi}}^{p}(\mathcal{X})}^{p}=\alpha^{p}\int_{\mathcal{X}}\left|\int_{-\infty}^{-T}\int_{\mathcal{X}}{Z}_{r,0}^{\ssup\beta}(y,x)\,\dif V_{r}(y)\right|^{p}w_{\mathcal{X};\xi}(x)^{p}\,\dif x.
\]
Taking expectation with respect to $(V_{r}(y))_{r,y}$ (and recalling that $\mathscr{F}$ is the $\sigma$-algebra generated by $(\dif U_r)_{r\in\mathbb{R}}$)
, we get
\begin{equation}\label{eq:takecondexp}\begin{aligned}
&\mathbb{E}\bigg[\|\tilde{v}-v_{0}^{\ssup T}\|_{L_{w_{\mathcal{X};\xi}}^{p}(\mathcal{X})}^{p}\bigg|\mathscr{F}\bigg]
\\
&=\alpha^{p}\int_{\mathcal{X}}\mathbb{E}\bigg[\,\, \bigg|\int_{-\infty}^{-T}\int_{\mathcal{X}}{Z}_{r,0}^{\ssup\beta}(y,x)\,\dif V_{r}(y)\bigg|^{p} \,\, \bigg|\,\, \mathscr{F}\bigg]w_{\mathcal{X};\xi}(x)^{p}\,\dif x
\\
&=C_p\alpha^{p}\int_{\mathcal{X}}\left(\int_{-\infty}^{-T}\iint_{\mathcal{X}^{2}}{Z}_{r,0}^{\ssup\beta}(y_{1},x){Z}_{r,0}^{\ssup\beta}(y_{2},x)R(y_{1}-y_{2})\,\dif y_{1}\,\dif y_{2}\,\dif r\right)^{p/2}w_{\mathcal{X};\xi}(x)^{p}\,\dif x,
\end{aligned}
\end{equation}
where in the second identity we used the Gaussianity of the noise, and the constant $C_p<\infty$ depends only on $p$.
For $T\ge1$, we can estimate the inside integral in \cref{eq:takecondexp}
by
\begin{align*}
&\int_{-\infty}^{-T}  \iint_{\mathcal{X}^{2}}{Z}_{r,0}^{\ssup\beta}(y_{1},x){Z}_{r,0}^{\ssup\beta}(y_{2},x)R(y_{1}-y_{2})\,\dif y_{1}\,\dif y_{2}\,\dif r\\
 &\ \overset{\cref{eq:applyCK}}{\le}\left(\sup_{z\in\mathcal{X}}\frac{{Z}_{-1,0}^{\ssup\beta}(z,x)}{w_{\mathcal{X};\xi/2}}(z)\right)^{2}\int_{-\infty}^{-T}\iint_{\mathcal{X}^{2}}\left(\prod_{i=1}^{2}\int_{\mathcal{X}}{Z}_{r,-1}^{\ssup\beta}(y_{i},z_{i})w_{\mathcal{X};\xi/2}(z_{i})\,\dif z_{i}\right)R(y_{1}-y_{2})\,\dif y_{1}\,\dif y_{2}\,\dif r\\
 &\ \overset{\cref{eq:QVrelation}}{=}\left(\sup_{z\in\mathcal{X}}\frac{{Z}_{-1,0}^{\ssup\beta}(z,x)}{w_{\mathcal{X};\xi/2}}(z)\right)^{2}\left([\tilde{M}(w_{\mathcal{X};\xi/2})]_{\infty}-[\tilde{M}(w_{\mathcal{X};\xi/2})]_{T}\right),
\end{align*}
where $\tilde{M}_{T}(\mu)\coloneqq\int_{\mathcal{X}}u_{-1}^{\ssup T}(x)\,\dif\mu(x)$
has the same distribution as $M_{T}(\mu)$. In the last identity of the above display, we have shifted time by $1$ in order to apply \cref{eq:QVrelation}, which applies equally well to $\tilde{M}$ as to $M$. Using the last estimate
in \cref{eq:takecondexp}, we get
\begin{align*}
    \mathbb{E}[\|\tilde{v}-v_{0}^{\ssup T}\|_{L_{w_{\mathcal{X};\xi}}^{p}(\mathcal{X})}^{p}\mid\mathscr{F}] & \le C_p \alpha^{p}\left([\tilde{M}(w_{\mathcal{X};\xi/2})]_{\infty}-[\tilde{M}(w_{\mathcal{X};\xi/2})]_{T}\right)^{p/2}\\
 & \qquad\times\int_{\mathcal{X}}\left(w_{\mathcal{X};\xi}(x)\sup_{z\in\mathcal{X}}\frac{{Z}_{-1,0}^{\ssup\beta}(z,x)}{w_{\mathcal{X};\xi/2}}(z)\right)^{p}\,\dif x.
\end{align*}
The last integral is finite a.s. by \cref{eq:Zpinftybd}. By \cref{eq:QVconverges}
(with time shifted by $-1$), this implies that 
\[
\lim_{T\to\infty}\mathbb{E}[\|\tilde{v}-v_{0}^{\ssup T}\|_{L_{w_{\mathcal{X};\xi}}^{p}(\mathcal{X})}^{p}\mid\mathscr{F}]=0\qquad\text{a.s.}
\]
By the conditional Markov inequality, this means that, for any fixed
$\eps>0$, we have
\begin{align*}
    \mathbb{P}[\|\tilde{v}-v_{0}^{\ssup T}\|_{L_{w_{\mathcal{X};\xi}}^{p}(\mathcal{X})}\ge\eps\mid\mathscr{F}] & =\mathbb{P}[\|\tilde{v}-v_{0}^{\ssup T}\|_{L_{w_{\mathcal{X};\xi}}^{p}(\mathcal{X})}^{p}\ge\eps^{p}\mid\mathscr{F}]\\&\le\eps^{-p}\mathbb{E}[\|\tilde{v}-v_{0}^{\ssup T}\|_{L_{w_{\mathcal{X};\xi}}^{p}(\mathcal{X})}^{p}\mid\mathscr{F}],
\end{align*}
so, for each $\eps>0$, we have
\[
\lim_{T\to\infty}\mathbb{P}[\|\tilde{v}-v_{0}^{\ssup T}\|_{L_{w_{\mathcal{X};\xi}}^{p}(\mathcal{X})}\ge\eps\mid\mathscr{F}]=0\qquad\text{a.s.}
\]
By the bounded convergence theorem, this means that 
\[
\lim_{T\to\infty}\mathbb{P}\big[\|\tilde{v}-v_{0}^{\ssup T}\|_{L_{w_{\mathcal{X};\xi}}^{p}(\mathcal{X})}\ge\eps\big]=\lim_{T\to\infty}\mathbb{E}\left[\mathbb{P}\big[\|\tilde{v}-v_{0}^{\ssup T}\|_{L_{w_{\mathcal{X};\xi}}^{p}(\mathcal{X})}\ge\eps\mid\mathscr{F}\big]\right]=0,
\]
and hence \cref{eq:convinprob} holds.

To prove \cref{eq:convinLp}, we use \cref{eq:QVrelation}
in \cref{eq:takecondexp} to see that
\[
    \mathbb{E}\bigg[\|\tilde{v}-v_{0}^{\ssup T}\|_{L_{w_{\mathcal{X};\xi}}^{p}(\mathcal{X})}^{p}\ \bigg|\ \mathscr{F}\bigg]=C_p\alpha^{p}\int_{\mathcal{X}}\big([M(\delta_{x})]_{\infty}-[M(\delta_{x})]_{T}\big)^{p/2}w_{\mathcal{X};\xi}(x)^{p}\,\dif x.
\]
Taking expectations and using the spatial stationarity of the noise,
we obtain
\[
\mathbb{E}\big[\|\tilde{v}-v_{0}^{\ssup T}\|_{L_{w_{\mathcal{X};\xi}}^{p}(\mathcal{X})}^{p}\big]=C_p\alpha^{p}\left(\int_{\mathcal{X}}w_{\mathcal{X};\xi}(x)^{p}\,\dif x\right)\mathbb{E}\big([M(\delta_{0})]_{\infty}-[M(\delta_{0})]_{T}\big)^{p/2}.
\]
The integral is finite by the definition \cref{eq:weightednormdef}
of $w_{\mathcal{X};\xi}$, and the expectation goes to $0$ as $T\to\infty$
by \cref{eq:LpQVconvergence} (which holds by the assumption
that $\beta\in I_{L^{p}}$).
\end{proof}

\begin{prop}
\label{prop:stationarity}The random function $\tilde{v}$ defined
in \cref{eq:vtildedef} is statistically stationary for the SPDE
\cref{eq:theSPDE}.
\end{prop}

\begin{proof}
Let $(\tilde{v}_{t}(x))_{t\ge0,x\in\mathcal{X}}$ solve \cref{eq:theSPDE}
with initial condition $\tilde{v}_{0}=\tilde{v}$. By \cref{eq:convinprob},
we have for any $\eps>0$ and $\xi>0$ that
\begin{equation}
    \lim_{T\to\infty}\mathbb{P}\left(\|\tilde{v}_{0}-v_{0}^{\ssup T}\|_{L_{w_{\mathcal{X};\xi}}^{p}(\mathcal{X})}\ge\eps\right)=0.\label{eq:compare-with-stationary}
\end{equation}
Subtracting two copies of \cref{eq:Duhamel}, we have
\[
    (\tilde v_t-v_t^{\ssup T})(x) = \int_{\mathcal{X}}{Z}^{\ssup\beta}_{0,t}(y,x)(\tilde v_0-v_0^{\ssup T})(y)\,\dif y.
\]
From this we can derive
\begin{align*}
    \|&\tilde{v}_{t}-v_{t}^{\ssup T}\|_{L_{w_{\mathcal{X};2\xi}}^{p}(\mathcal{X})}\\	&=\left(\int_{\mathcal{X}}\left|\int_{\mathcal{X}}{Z}_{0,t}^{\ssup\beta}(y,x)(\tilde{v}_{0}-v_{0}^{\ssup T})(y)\,\dif y\right|^{p}w_{\mathcal{X};2\xi}(x)^{p}\,\dif x\right)^{1/p}\\
                                                                                &\le\left(\int_{\mathcal{X}}\left(\int_{\mathcal{X}}\left(\frac{{Z}_{0,t}^{\ssup \beta}(y,x)}{{w}_{\mathcal{X};\xi}(y)}\right)^{q}\,\dif y\right)^{p/q}\left(\int_{\mathcal{X}}\left|(\tilde{v}_{0}-v_{0}^{\ssup T})(y)\right|^{p}{w}_{\mathcal{X};\xi}(y)^{p}\,\dif y\right)w_{\mathcal{X};2\xi}(x)^{p}\,\dif x\right)^{1/p}\\
                                                                                &=\left(\int_{\mathcal{X}}\left(\int_{\mathcal{X}}\left(\frac{{Z}_{0,t}^{\ssup\beta}(y,x)}{{w}_{\mathcal{X};\xi}(y)}\right)^{q}\,\dif y\right)^{p/q}w_{\mathcal{X};2\xi}(x)^{p}\,\dif x\right)^{1/p}\|\tilde{v}_{0}-v_{0}^{\ssup T}\|_{L_{w_{\mathcal{X};\xi}}^{p}(\mathcal{X})},
    \end{align*}
    The last integral is finite a.s. by \cref{eq:Zpqbd}.
    This along with \cref{eq:compare-with-stationary} implies
that for every $\eps>0$ and $t\ge0$ we have 
\[
\lim_{T\to\infty}\mathbb{P}\left(\|\tilde{v}_{t}-v_{t}^{\ssup T}\|_{L_{w_{\mathcal{X};2\xi}}^{p}(\mathcal{X})}\ge\eps\right)=0.
\]
We note that \[v_{t}^{\ssup T}\overset{\mathrm{law}}{=}v_{0}^{\ssup {T+t}}\]
and that
\[
\lim_{T\to\infty}\mathbb{P}\left(\|v_{0}^{\ssup T}-v_{0}^{\ssup{T+t}}\|_{L_{w_{\mathcal{X};2\xi}}^{p}(\mathcal{X})}\ge\eps\right)=0
\]
by \cref{eq:convinprob}. Combining the last three displays,
we see that $\tilde{v}_{0}\overset{\mathrm{law}}{=}\tilde{v}_{t}$,
which implies the desired stationarity.
\end{proof}

Now we can prove \cref{thm:convergence,thm:stsproperties}.
\begin{proof}[Proof of \cref{thm:convergence,thm:stsproperties}]
    \cref{thm:convergence} is the combination of \cref{eq:convinprob,prop:stationarity}.
    The statement \cref{eq:Lpstatement} of \cref{thm:stsproperties} has been proved as \cref{eq:convinLp} of \cref{prop:Lpconvergence}.
    Finally, the statement \cref{eq:Cauchystatement} of \cref{thm:stsproperties} follows from \cref{eq:limitingdistn-1,eq:vtildedef} along with \cref{prop:Cauchylimits} applied with $\mu = \sum_{i=1}^n \gamma_i\delta_{x_i}$: the assumption that $M_\mu(\infty)=0$ is satisfied because $M_\mu(T) = \sum_{i=1}^n \gamma_i M_{\delta_{x_i}}(T)$ and each of the finitely many terms on the right side goes to $0$ as $T\to\infty$ by the assumption.
\end{proof}

\appendix
    \section{The phase transition}
    \label{sec:phase-transition}
In this section we give references for the statements in \cref{prop:phase-transition}.
In the case $d\ge 3$ and $\mathcal{X} = \mathbb{R}^d$, \crefrange{enu:weak}{enu:dge3} of the proposition have been proved as \cite[Theorems~2.1 and~2.4 and Corollary~2.5]{MSZ16}, and \cref{enu:ILprelation} has been proved as \cite[Theorem 2.3]{RLM22}, which extends the work \cite{Jun22b} in the discrete setting.

In the remaining cases, the claim is that $I_{\mathrm{strong}} = (0,\infty)$, and to prove the proposition in this case it is sufficient to show that $u^{\ssup T}_0(0)$ converges in distribution to $0$ as $T\to\infty$, or equivalently that $u^{\ssup 0}_t(0)$ converges to $0$ as $t\to \infty$.
In the case $\mathcal{X}=(\mathbb{R}/\mathbb{Z})^d$, this is a consequence of \cite[Theorem~1.1]{GK23}. In the case $R=\delta$ and $\mathcal{X}=\mathbf{R}$, then it can be seen from \cite[Theorem~2.2]{QS23} (where the relevant piece is that the Itô factor must be added to get convergence: without this addition the solution goes to zero). See also \cite{LDC12,Vir20}.

In the case when $R$ is smooth and $\mathcal{X} = \mathbf{R}^d$ for $d=1,2$, we are not aware of a proof of this precise statement in the literature. The problem is essentially the same as the Brownian polymer with Poisson forcing considered in \cite{CY05}. For completeness, we sketch a proof of the statement here, following the proof of \cite[Theorem~2.1.1(b)]{CY05}.

\begin{proof}[Proof of \cref{prop:phase-transition}\textup{(\ref{enu:dle2})} in the case $\mathcal{X}=\mathbb{R}^d$, $d=1,2$]
Let \[M_T \coloneqq M_T(\delta_0) = u_0^{\ssup T}(0),\] with the middle quantity defined in \cref{eq:Mdef}.
    Fix $\theta\in (0,1)$ and define $m(T) \coloneqq \mathbb{E}[M_T^\theta]$. By \cref{eq:MmuIto} and Itô's formula, we can compute
    \begin{align}
        m'(T) &= -\frac12\beta^2\theta(1-\theta)\mathbb{E}\left[M_T^{\theta-2}\iint_{\mathcal{X}^2}Z^{\ssup\beta}_{-T,0}(y_1,0)Z^{\ssup\beta}_{-T,0}(y_2,\beta)R(y_1-y_2)\,\dif y_1\,\dif y_2\right]\notag\\
              &= -\frac12\beta^2\theta(1-\theta)\mathbb{E}\left[M_T^{\theta}\frac{\iint_{\mathcal{X}^2}Z^{\ssup\beta}_{-T,0}(y_1,0)Z^{\ssup\beta}_{-T,0}(y_2,\beta)R(y_1-y_2)\,\dif y_1\,\dif y_2}{\left(\int_{\mathcal{X}}Z^{\ssup\beta}_{-T,0}(y,0)\,\dif y\right)^2}\right]\label{eq:writemprime}
\end{align}
Fix $\ell_T>\operatorname{diam}\supp\phi$, to be chosen later, and let $\Lambda_T$ be the ball of radius $\ell_T$ about the origin in $\mathbb{R}^d$.  %
Let $|\Lambda_T|$ denote the Lebesgue measure of $\Lambda_T$. We can estimate the integral as
\begin{align*}
    \iint_{\mathcal{X}^2}&Z^{\ssup\beta}_{-T,0}(y_1,0)Z^{\ssup\beta}_{-T,0}(y_2,\beta)R(y_1-y_2)\,\dif y_1\,\dif y_2\\&\ge \iint_{\Lambda_T^2} Z^{\ssup\beta}_{-T,0}(y_1,0)Z^{\ssup\beta}_{-T,0}(y_2,\beta)R(y_1-y_2)\,\dif y_1\,\dif y_2\\
                                                                                                                   &\ge \frac1{|\Lambda_T|}\left(\iint_{\Lambda_T^2} Z_{-T,0}^{\ssup\beta}(y,0)\phi(y-z)\,\dif z\,\dif y\right)^2\\
                                                                                                                   &\ge \frac{\left(\int_{\mathcal{X}}\phi(z)\,\dif z\right)^2}{|\Lambda_T|}\left(\int_{\widetilde{\Lambda}_T} Z_{-T,0}^{\ssup\beta}(y,0)\,\dif y\right)^2,
\end{align*}
where we have defined $\widetilde\Lambda_T$ to be the ball of radius $\ell_T-\operatorname{diam}\supp\phi$ about the origin in $\mathbb{R}^d$.
From this we estimate
\begin{align*}
    &\frac{\iint_{\mathcal{X}^2}Z^{\ssup\beta}_{-T,0}(y_1,0)Z^{\ssup\beta}_{-T,0}(y_2,\beta)R(y_1-y_2)\,\dif y_1\,\dif y_2}{\left(\int_{\mathcal{X}}Z^{\ssup\beta}_{-T,0}(y,0)\,\dif y\right)^2}\\
    &\ge \frac{\left(\int_{\mathcal{X}}\phi(z)\,\dif z\right)^2}{|\Lambda_T|}\left(1-\frac{\int_{\widetilde{\Lambda}_T^{\mathrm{c}}} Z_{-T,0}^{\ssup\beta}(y,0)\,\dif y}{\int_{\mathcal{X}}Z^{\ssup\beta}_{-T,0}(y,0)\,\dif y}\right)^2\\
    &\ge\frac{\left(\int_{\mathcal{X}}\phi(z)\,\dif z\right)^2}{|\Lambda_T|}\left(1-2\left(\frac{\int_{\widetilde{\Lambda}_T^{\mathrm{c}}} Z_{-T,0}^{\ssup\beta}(y,0)\,\dif y}{\int_{\mathcal{X}}Z^{\ssup\beta}_{-T,0}(y,0)\,\dif y}\right)^\theta\right).
\end{align*}
Using this in \cref{eq:writemprime}, we get
\begin{align*}
    m'(T)&\le-\frac{\beta^2\theta(1-\theta)\left(\int_{\mathcal{X}}\phi(z)\,\dif z\right)^2}{2|\Lambda_T|}\mathbb{E}\left[M_T^{\theta} - 2\left(M_T \frac{\int_{\widetilde{\Lambda}_T^{\mathrm{c}}} Z_{-T,0}^{\ssup\beta}(y,0)\,\dif y}{\int_{\mathcal{X}}Z^{\ssup\beta}_{-T,0}(y,0)\,\dif y}\right)^\theta\right]\\
         &\le \frac{\beta^2\theta(1-\theta)\left(\int_{\mathcal{X}}\phi(z)\,\dif z\right)^2}{2|\Lambda_T|}\left[-m(T)+2\left(\mathbb{E}\int_{\widetilde{\Lambda}_T^{\mathrm{c}}} Z_{-T,0}^{\ssup\beta}(y,0)\,\dif y\right)^\theta\right]\\
         &= \frac{\beta^2\theta(1-\theta)\left(\int_{\mathcal{X}}\phi(z)\,\dif z\right)^2}{2|\Lambda_T|}\left[-m(T)+2(\mathrm{P}(|X(-T)|\ge \ell_T-\operatorname{diam}\supp\phi))^\theta\right],
\end{align*}
where $P$ denotes the probability measure of a standard two-sided Brownian motion $X$ with $X(0)=0$.
Therefore, the conclusion of \cite[Lemma~4.12]{CY05} holds, except that the ball $U(x)$ there is replaced by a ball of radius $\operatorname{diam}\supp\phi$, which makes no difference. The rest of the proof then proceeds exactly as on \cite[p.~271]{CY05} modulo this replacement.
\end{proof}

\bibliographystyle{hplain-ajd}
\bibliography{citations}

\end{document}